\documentclass[12 pt]{amsart}
\usepackage{amsmath,amssymb,amsthm,hyperref}
\usepackage{longtable,}
\usepackage{a4wide}
\usepackage{color}
\usepackage{enumerate}
\usepackage{cleveref}

\usepackage{tikz}
\usetikzlibrary{decorations.pathreplacing}

\newtheorem{thm}{Theorem}[section]
\newtheorem*{thm*}{Theorem}
\newtheorem*{conj*}{Conjecture}
\newtheorem{cor}[thm]{Corollary}

\newtheorem{prop}[thm]{Proposition}

\theoremstyle{remark}

\theoremstyle{definition}

\newtheorem{defn}[thm]{Definition}
\newcounter{claim}[thm]

\newcommand{\Z}{\mathbb{Z}} 

\DeclareMathOperator{\Cay}{Cay}

\crefname{lem}{lemma}{lemmas}
\Crefname{lem}{Lemma}{Lemmas}
\crefmultiformat{lem}{Lemmas~#2#1#3}{ and~#2#1#3}{, #2#1#3}{, and~#2#1#3}

\title{Two new families of non-CCA groups}

\author{Brandon Fuller}
\author{Joy Morris}

\address{Department of Mathematics and Computer Science\\
University of Lethbridge\\
Lethbridge, AB T1K 3M4\\
Canada}\thanks{This work was supported by the Natural Science and Engineering Research Council of Canada (grant RGPIN-2017-04905).}

\email{joy.morris@uleth.ca}\email{brandon.fuller621@gmail.com}	

\begin{document}

\begin{abstract}
We determine two new infinite families of Cayley graphs that admit colour-preserving automorphisms that do not come from the group action. By definition, this means that these Cayley graphs fail to have the CCA (Cayley Colour Automorphism) property, and the corresponding infinite families of groups also fail to have the CCA property. The families of groups consist of the direct product of any dihedral group of order $2n$ where $n \ge 3$ is odd, with either itself, or the cyclic group of order $n$. In particular, this family of examples includes the smallest non-CCA group that had not fit into any previous family of known non-CCA groups.
\end{abstract}

\maketitle

\section{Introduction}
All groups and graphs in this paper are finite. All of our graphs are simple, undirected, and have no loops.

A Cayley graph of $G$ with respect to $C$ (a subset of $G \backslash \{e\}$) is the graph $Cay(G,C)$ whose vertices are the elements of $G$, with an edge from $g$ to $gc$ if and only if $g\in G, c\in C$.
The set $C$ is known as the connection set of $Cay(G,C)$. This connection set gives a natural colouring of the edges where we colour the edge from $g$ to $gc$ (which is the same as the edge from $gc$ to $g$) with a colour associated to
$\{c,c^{-1}\}$. A colour-preserving automorphism of $Cay(G,C)$ is a permutation of the vertices that preserves edges and non-edges as well as edge colour. A Cayley graph $Cay(G,C)$ is said to have the 
\textbf{Cayley Colour Automorphism (CCA) property} if every colour-preserving automorphism of the graph is an affine function on $G$. The group $G$ is said to be CCA if every connected Cayley graph of $G$ is CCA.

The study of this property has only come up recently in history. In 2012, M.~Conder, T.~Pizanski and A.~\v{Z}itnik \cite{gi-graphs} proposed a question 
about the permutations on circulant graphs that preserved a certain edge colouring that the second author 
\cite{aut-circ-part} answered. The second author showed that for any connected circulant graph  all 
colour-preserving automorphisms that fix the identity are automorphisms of $\Z_n$. In 2014, A.~Hujdurovi\'c, K.~Kutnar, D.~W.~Morris, and J.~Morris \cite{col-pres} extended 
the original question by looking at Cayley graphs, using the natural edge colouring described. 
In early 2017, L.~Morgan, J.~Morris and G.~Verret \cite{sylow-cyclic}, \cite{cca-simple} gave some new results for finite simple groups and Sylow cyclic groups that generalized results produced by E.~Dobson, A.~Hujdurovi\'c, K.~Kutnar, and J.~Morris in \cite{odd-square-free}. The problem of determining colour-preserving and colour-permuting automorphisms for directed Cayley graphs has 
already been studied and is well understood: see for example 
 \cite{ggs}, where the authors
showed that for every connected Cayley digraph, every colour-preserving automorphism is a left-translation by some element of the group. 

In his M.Sc. thesis, the first author produced code that determines whether or not a group or graph has the CCA property, and ran this code on all groups of order up to order 200 (excluding orders 128 and 192). With this data in hand, a logical step was to try to find theoretical methods to explain some of the small non-CCA groups that were not previously understood, and if possible to find new infinite families of non-CCA groups using this method.

In this paper, we will use results from \cite{sylow-cyclic} to show that whenever $n \ge 3$, the groups $C_n \times D_{2n}$ and $D_{2n}\times D_{2n}$ are non-CCA groups. 
Section~\ref{background} will contain some basic background, definitions, and notation, along with the statements of the results we need from~\cite{sylow-cyclic}. 
Section~\ref{proofs} will provide proofs of our main results.

\section{Background}\label{background}


We will use the following notation for the remainder of this paper. We use $C_n$ to represent the cyclic group of order $n$ and $D_{2n}$ (for $n \geq 3$) to represent the dihedral group of order $2n$. 
We also have $Q_8$ as the quaternion group of order 8. 

The notation $\Gamma = (V(\Gamma), E(\Gamma))$ will represent a graph of finite order, consisting of a set 
$V = V(\Gamma)$ of vertices and a set $E=E(\Gamma) \subseteq V \times V$ of edges. The set of vertices that are adjacent to a vertex $v$, denoted $N(v)$, is called
the neighbours of $v$. We use $\mathcal{L}(\Gamma)$ to indicate the line graph of the graph $\Gamma$, and $\mathcal{S}(\Gamma)$ is the subdivision graph of 
the graph $\Gamma$.

If $G$ acts on a graph $\Gamma$ and $S \subseteq V(\Gamma)$ then $G^S$ is the restriction of the action of $G$ to $S$.
We use $G_{v}$ to denote the stabiliser subgroup (elements of $G$ that fix $v$).

\begin{defn} \cite[Definition.~2.6]{col-pres} \label{def:gendic}
For an abelian group $A$ of even order and an involution $y \in A$, the corresponding \textbf{generalized dicyclic group} is 
\begin{displaymath}
Dic(A,y) = \langle x, A \mid x^2 = y, x^{-1}ax = a^{-1} \forall a \in A \rangle. 
\end{displaymath}
\end{defn}

\begin{defn} \cite[Definition.~5.1]{col-pres} \label{def:gendih}
 The \textbf{generalized dihedral group} over an abelian group $A$ is the group
\begin{displaymath}
 Dih(A) = \langle \sigma, A \mid \sigma^2 = e, \sigma a \sigma = a^{-1} \forall a \in A \rangle
\end{displaymath}
\end{defn}

\begin{defn} [{}{\cite[Definition~4.5]{sylow-cyclic}}] \label{def:ccp}
 Let $B$ be a permutation group and $G$ a regular subgroup of $B$. Let $\mathcal{A}^0$ be the colour-preserving automorphism group for the complete Cayley colour graph $K_G=\Cay(G,G\setminus\{e\})$, and 
 let $\widehat{G}$ be the subgroup of $\mathcal{A}^0$ consisting of all left translations by elements of $G$. We say that $(G,B)$ is a \textbf{complete colour pair} if $B$ is 
 a subgroup of $\mathcal{A}^0$ and $G$ is one of the following:
\begin{itemize}
 \item $G$ is abelian but not an elementary abelian 2-group, and $\mathcal{A}^0 = Dih(G)$.
 \item $G\cong Dic(A,y)$ but not of the form $Q_8 \times C^n_2$ and $\mathcal{A}^0 = \widehat{G} \rtimes \langle \sigma \rangle$, where $\sigma$ is the permutation that fixes $A$
pointwise and maps every element of the coset $Ax$ to its inverse.
 \item $G\cong Q_8 \times C^n_2$ and $\mathcal{A}^0 = \langle \widehat{G}, \sigma_i, \sigma_j, \sigma_k \rangle$, where $\sigma_\alpha$ is the permutation of 
$Q_8 \times C^n_2$ that inverts every element of $\{ \pm \alpha \} \times C^n_2$ and fixes every other element.
\end{itemize}
\end{defn}

The importance of Definition \ref{def:ccp} comes from the fact that if $(G,B)$ is a complete colour pair, then in each case we have a colour-preserving automorphism
of $K_G$ that is not an element of $\widehat{G}$.

An \emph{arc} is an orientation for an edge in a graph. So the edge $\{u,v\}$ contains two arcs: $(u,v)$, or $(v,u)$, depending on the orientation we choose.

\begin{defn} [{}{\cite[Definition~1.1]{gtac}}]
  Let $\Gamma$ be a graph and $G$ a permutation group acting on the vertices of $\Gamma$. We say that $\Gamma$ is a \textbf{$G$-arc-regular graph} if for each pair of arcs 
  $e_1=(u,v)$ and $e_2=(w,x)$ (each an oriented edge from $E(\Gamma)$), there exists a unique element of $G$ that maps $u$ to $w$ and $v$ to $x$, so that it maps the chosen orientation for $e_1$ to the chosen orientation for $e_2$.
\end{defn}

\begin{cor} [{}{\cite[Corollary~4.10]{sylow-cyclic}}] \label{cor:argls}
 Let $\Gamma$ be a connected $G$-arc-regular graph. If $H$ is a group of automorphisms of $\Gamma$ such that:
 \begin{itemize}
  \item $G \leq H$, and
  \item $(G^{N(v)}_v, H^{N(v)}_v)$ is a complete colour pair for every vertex $v$ of $\Gamma$,
 \end{itemize}
  then $H$ is a colour-preserving group of automorphisms of $\mathcal{L}(\mathcal{S}(\Gamma))$ viewed as a Cayley graph on $G$.
\end{cor}

The real point of this corollary is that if we can show that some element of $H$ is not an affine function, then we can conclude that $\mathcal{L}(\mathcal{S}(\Gamma))$ is a non-CCA graph, and so $G$ is a non-CCA group. The fact that $(G^{N(v)}_v, H^{N(v)}_v)$ is a complete colour pair is what allows us to produce the non-affine automorphism.

\section{Main Results}\label{proofs}

In our main result, we will show that $K_{n,n}$ is a (connected) $C_n \times D_{2n}$-arc-regular graph and therefore that if we take $\Gamma=K_{n,n}$, $G=C_n \times D_{2n}$, and $H=D_{2n} \wr C_2$ then all of the conditions of Corollary \ref{cor:argls} are satisfied. For clarity (since both orders are used often in the literature around wreath products), by $D_{2n} \wr C_2$ we intend the semidirect product $(D_{2n} \times D_{2n}) \rtimes C_2$, where the $C_2$ is acting on the coordinates in the direct product.  Hence $D_{2n} \wr C_2$
is a colour-preserving group of automorphisms of $\mathcal{L}(\mathcal{S}(K_{n,n}))$. With this we will find an element in $D_{2n} \wr C_2$, a colour-preserving automorphism, that is a 
non-affine function to show that $\mathcal{L}(\mathcal{S}(K_{n,n}))$ is non-CCA. The proof is not particularly difficult; the difficulty of this result lies in finding an arc-regular graph and corresponding permutation groups to which we can apply Corollary~\ref{cor:argls}.

\begin{thm} \label{thm:cnxd2n}
 The graph $\mathcal{L}(\mathcal{S}(K_{n,n}))$ viewed as a Cayley graph on $C_n \times D_{2n}$ is non-CCA whenever $n \geq 3$ is odd. 
 
 Specifically, if $G=\langle \rho_1, \rho_2, \tau\rangle$ and $C=\{\tau\} \cup \{\rho_2^i: 1 \le i \le n-1\}$, then $\sigma_2$ is a non-affine colour-preserving automorphism on $\Cay(G,C)$.
\end{thm}
\begin{proof}
 Consider the complete bipartite graph $K_{n,n}$. 
 Let $\rho_1$ be a cyclic permutation on one of the bipartition sets, and $\rho_2$ be a cyclic permutation on the other bipartition set, with $\tau$ an involution that commutes with $\rho_1\rho_2$ and switches the bipartition sets.
 Observe that that $G = \langle \rho_1,\rho_2,\tau \rangle=\langle \rho_1\rho_2,\rho_1\rho_2^{-1},\tau \rangle\cong C_n \times D_{2n}$ since $n$ is odd so that $\langle\rho_2^2\rangle=\langle\rho_2\rangle$. Notice that $G$ acts regularly 
 on the arcs of $K_{n,n}$, so that $K_{n,n}$ is $G$-arc-regular.
 
Let $\sigma_1$ be an involution acting on the first of the bipartition sets that inverts $\rho_1$, and define $\sigma_2$ similarily for the other bipartition set.
 Consider now the group $H = \langle \rho_1, \rho_2, \tau, \sigma_1, \sigma_2 \rangle \cong D_{2n} \wr C_2$ where each copy of $D_{2n}$ acts independently on one of the bipartition sets of $K_{n,n}$, and the $C_2$ (generated by $\tau$) exchanges the coordinates. 
 The first copy of $D_{2n}$ is generated by $\rho_1$ and $\sigma_1$. The second copy is generated by 
 $\rho_2$  and $\sigma_2$. 
 It is clear that $G \leq H$ since $\rho_1, \rho_2, \tau \in H$. 
 
 If we can show that $(G^{N(v)}_v, H^{N(v)}_v)$ is a complete colour pair for an arbitrary vertex $v$ then we
 can use Corollary~\ref{cor:argls} to say that $H$ is a colour-preserving group of automorphisms of $\mathcal{L}(\mathcal{S}(K_{n,n}))$ viewed as a Cayley graph on $G$.

 Let $v$ be an arbitrary vertex of one of the bipartition sets. The neighbours of $v$ are all the elements of the other bipartition set. We notice that $G^{N(v)}_v$ is the subgroup
 of $G$ that fixes $v$ and its action is restricted to the bipartition set that $v$ is not in. Without loss of generality, suppose $\rho_1$ is the cyclic permutation of $N(v)$.
 Since $G = \langle \rho_1,\rho_2,\tau \rangle$, it is not hard to observe that $G^{N(v)}_v = \langle \rho_1 \rangle \cong C_n$. Similarily since 
 $H = \langle \rho_1, \rho_2, \tau, \sigma_1, \sigma_2 \rangle$ we have that $H^{N(v)}_v = \langle \rho_1, \sigma_1 \rangle \cong D_{2n}$. Thus we only need to show that
 $(C_n, D_{2n})$ is a complete colour pair.
 
 We can see $(C_n, D_{2n})$ is a complete colour pair using Definition \ref{def:ccp}. Let $\mathcal{A}^0$ be the colour-preserving automorphism group for the Cayley graph 
 $K_G$. We know that $\mathcal{A}^0 = D_{2n} = Dih(C_n)$ and thus since $C_n$ is abelian and is not an elementary abelian $2$-group ($n \ge 3$), all the properties of the first possibility for a complete colour pair are met. (In this case, $B=D_{2n}=\mathcal A^0$.) We thus conclude (using Corollary~\ref{cor:argls}) that every element of $H$ is a colour-preserving automorphism of $\mathcal{L}(\mathcal{S}(K_{n,n}))$ viewed as a Cayley graph on $G$.

It remains to show that some element of $H$ is not affine; in other words, that it fails to normalise $G$. We claim that $\sigma_2$ is such an element. In order to prove this, we will show that $\sigma_2^{-1}\tau\sigma_2$ is not an element of $G$. Let $v$ be the unique vertex in the second bipartition set that is fixed by $\sigma_2$. Clearly, $\sigma_2^{-1}\tau\sigma_2=\sigma_2\tau\sigma_2$ maps the arc $(v, \tau(v))$ to the arc $(\tau(v),v)$, since $\sigma_2$ fixes both $v$ and $\tau(v$). Since $G$ is acting arc-regularly, it has a unique element that maps $(v, \tau(v))$ to the arc $(\tau(v),v)$, and we know that this element is $\tau$. So if $\sigma_2$ normalises $G$, we must have $\sigma_2\tau\sigma_2=\tau$. It is straightforward to verify that this is not the case. For example, 
$\sigma_2 \tau\sigma_2\rho_2(v)=\tau\sigma_2 \rho_2(v)=\tau\rho_2^{-1}\sigma_2(v)=\tau\rho_2^{-1}(v)$ (the first equality follows because $\sigma_2$ fixes the bipartition set that does not contain $v$; the second because $\langle \sigma_2, \rho_2\rangle \cong D_{2n}$, so $\sigma_2$ inverts $\rho_2$; and the third because $\sigma_2$ fixes $v$). However, since $n \ge 3$, this is not the same as $\tau\rho_2(v)$, because the order of $\rho_2$ is $n$. Thus, $\sigma_2 \in H$ does not normalise $G$, as claimed.

We have shown that $\sigma_2\in H$ is not an affine function on $G$, even though it is a colour-preserving automorphism which 
implies that $\mathcal{L}(\mathcal{S}(K_{n,n}))$ is a Cayley graph on $G$ that is not CCA.  As we state in the next corollary, this implies that $C_n \times D_{2n}$ is non-CCA.

If we label the edges of $\mathcal {S}(K_{n,n})$ by labelling the edge from $\tau(v)$ to the vertex subdividing $\{v,\tau(v)\}$ with the identity element $e$ of $G$, and labelling each other edge by the unique element of $G$ that maps the edge $e$ to that edge, this produces the labeling that shows us that $\mathcal{L}(\mathcal{S}(K_{n,n}))$ is a Cayley graph on $G$. From this it is straightforward to observe that the connection set $C$ (which consists of all neighbours of $e$) is as claimed.
 \end{proof}
 
 \begin{cor} \label{cor:cnxd2n}
  The group $C_n \times D_{2n}$ is non-CCA whenever $n \geq 3$ is odd.
 \end{cor}
 
We use the above result to show that $D_{2n} \times D_{2n}$ is not CCA whenever $n \ge 3$ is odd.

\begin{prop} \label{prop:cnxg}
 The group $D_{2n} \times D_{2n}$ is non-CCA whenever $n \geq 3$ is odd.
\end{prop}

\begin{proof}
Define $H=\langle G, \gamma\rangle$, where $\gamma$ commutes with $\tau$ and with $\rho_1^{-1}\rho_2$, and inverts $\rho_1\rho_2$. Notice that this implies $H \cong D_{2n}\times D_{2n}$.

By Theorem~\ref{thm:cnxd2n},  if $G=\langle \rho_1, \rho_2, \tau\rangle$ and $C=\{\tau\} \cup \{\rho_2^i: 1 \le i \le n-1\}$, then $\sigma_2$ is a non-affine automorphism on $\Cay(G,C)$. We will use this to produce a non-affine colour-preserving automorphism $\varphi$ on $\Gamma=\Cay(H, C \cup \{\gamma\})$.

Define $\varphi$ by $\varphi(g)=\sigma_2(g)$, and $\varphi(g\gamma)=\sigma_2(g)\gamma$  for every $g \in G$. Since $\sigma_2$ is not affine on $G$, it is straightforward to see that $\varphi$ is not affine on $H$. Indeed, we can use the same argument as in the proof of Theorem~\ref{thm:cnxd2n} to show that since $\varphi$ restricted to $G$ (which is $\sigma_2$) does not normalise $G$, neither can $\varphi$ as a whole. It remains to show that $\varphi$ is colour-preserving on $\Gamma$.

Consider any edge $e$ of $\Gamma$. If both endpoints of $e$ are in $G$ then $\varphi(e)=\sigma_2(e)$ and since $\sigma_2$ preserves colours, so does $\varphi$.

If one endpoint of $e$ is in $G$ and the other is not, then it must be the case that $e$ is coloured $\gamma$, and its endpoints are $g$ and $g\gamma$ for some $g \in G$. Furthermore, by definition of $\varphi$ we have $\varphi(g\gamma)=\varphi(g)\gamma$, so there is an edge between $\varphi(g)$ and $\varphi(g\gamma)$, and its colour is $\gamma$. Thus $\varphi$ also preserves the colour of any such edge.

The final case to consider is if both endpoints of $e$ are in $G\gamma$. Suppose the endpoints of $e$ are $\rho_1^{i_1}\rho_2^{i_2}\tau^e\gamma$ and $\rho_1^{j_1}\rho_2^{j_2}\tau^f\gamma$, where $0 \le i_1, i_2, j_1, j_2 \le n-1$, and $0 \le e,f \le 1$. Since there is an edge between these vertices, we must have $\gamma\tau^{e}\rho_1^{j_1-i_i}\rho_2^{j_2-i_2}\tau^f\gamma \in C$ (recall that $\gamma$ and $\tau$ are both involutions). Note that $\rho_1^a\rho_2^b=(\rho_1\rho_2)^{(a+b)/2}(\rho_1^{-1}\rho_2)^{(b-a)/2}$; we want to use this because we know that $\gamma$ commutes with $\tau$ and with $\rho_1^{-1}\rho_2$ but inverts $\rho_1\rho_2$. So we have \begin{eqnarray*}\gamma\tau^{e}(\rho_1\rho_2)^{(j_1+j_2-i_1-i_2)/2}(\rho_1^{-1}\rho_2)^{(j_2+i_1-i_2-j_1)/2}\tau^f\gamma &=&\tau^e(\rho_1\rho_2)^{(i_1+i_2-j_1-j_2)/2}(\rho_1^{-1}\rho_2)^{(j_2+i_1-i_2-j_1)/2}\tau^f \\&=&\tau^e \rho_1^{i_2-j_2}\rho_2^{i_1-j_1}\tau^f\in C.\end{eqnarray*}
Since we know the elements of $C$, this implies one of three possibilities: 
\begin{itemize}
\item the element is $\tau$, so that $i_2=j_2$ and $i_1=j_1$, and $\{e,f\}=\{0,1\}$; 
\item $e=f=0$ and the element is $\rho_2^j$ for some $1\le j \le n-1$, so $i_2=j_2$, and $j=i_1- j_1$); or
\item $e=f=1$ and the element is $\rho_2^j$ for some $1 \le j \le n-1$, so (using the above equation and the fact that $\tau$ commutes with $\rho_1\rho_2$ and inverts $\rho_1^{-1}\rho_2$) $i_1=j_1$, and $j=j_2 - i_2$.
\end{itemize}

We now need to understand the images of these endpoints under $\varphi$. These are $\rho_1^{i_1}\rho_2^{-i_2}\tau^e\gamma$ and $\rho_1^{j_1}\rho_2^{-j_2}\tau^f\gamma$. Now using similar calculations to those in the previous paragraph, the colour of the edge between these is 
$$\tau^e\rho_1^{j_2-i_2}\rho_2^{i_1-j_1}\tau^f$$ (together with its inverse).
Taking the three possibilities identified above in turn, if $i_1=j_1$, $i_2=j_2$, and $\{e,f\}=\{0,1\}$ then this colour is $\tau$ as before, so $\varphi$ has preserved the colour. If $e=f=0$, $i_2=j_2$, and the colour of $e$ was $\{\rho_2^j, \rho_2^{-j}\}$ where $j=i_1-j_1$, then the colour of this edge is also $\{\rho_2^j,\rho_2^{-j}\}$. Finally, if $e=f=1$, $i_1=j_1$, and the colour of $e$ was $\{\rho_2^j,\rho_2^{-j}\}$ where $j=j_2-i_2$, then the colour of this edge is $\{\rho_2^{j},\rho_2^{-j}\}$. So in all cases the colour of $e$ is preserved under the action of $\varphi$. This completes the proof.
\end{proof}

\end{document}